\documentclass{amsart}
\usepackage[style=alphabetic]{biblatex} 
\usepackage{tikz}
\usepackage{amssymb}
\usepackage{amsthm}
\usepackage[all]{xy}
\usepackage{microtype}
\usepackage{xcolor}
\usepackage{hyperref}
\usepackage[nameinlink]{cleveref}

\hypersetup{
 colorlinks,
 linkcolor={teal},
 citecolor={teal},
 urlcolor={teal}
}

\calclayout

\DeclareFieldFormat
  [article,book,inbook,incollection,inproceedings,patent,thesis,unpublished]
  {title}{\emph{#1\isdot}}

\theoremstyle{plain}
	\newtheorem{theorem}{Theorem}
	\newtheorem{lemma}[theorem]{Lemma}

\theoremstyle{definition} 
	
	\newtheorem{definition}[theorem]{Definition}

\def\R{\mathbb{R}}
\newcommand{\norm}[1]{\| #1 \|}

\begin{document}
\title{The bounds of rigid sphere design}
\author[Y. Kamio]{Yuhi Kamio}
\address{College of Arts and Sciences, University of Tokyo, 3-8-1 Komaba, Meguro-ku, Tokyo 153-8914, Japan}
\email{emirp13@g.ecc.u-tokyo.ac.jp}

\date{\today}
\thanks{}

\begin{abstract}
We have identified some necessary conditions for the existence of rigid sphere designs. In particular, we have successfully resolved the conjecture proposed by \cite{bannai_rigid}; Given fixed positive integers $t$ and $d$, we show that there exist only finitely many rigid $t$-designs on $S^d$, up to orthogonal transformations.
\end{abstract}

\maketitle
\setcounter{tocdepth}{1}

\enlargethispage*{20pt}
\thispagestyle{empty}
We say a design $X$ is \emph{$m$-sized}, if $\#X =m$. 

The following concept was introduced implicitly by \cite{bannai_orbit} and explicitly by \cite{bannai_rigid}. 
\begin{definition}\label{rigid_design}
    We say an $m$-sized $t$-design $\{v_1,v_2,\ldots, v_m\}=X\subset S^d $ is \emph{rigid} if there exists a positive $\varepsilon$ that satisfies the following condition:
    \begin{quote}
        If an $m$-sized $t$-design $\{w_1,w_2,\ldots ,w_m\} = Y\subset S^d$ satisfies 
        $\norm{w_i-v_i}\leq \varepsilon $ for all $i$, then there exists a $\sigma \in SO(n)$ such that $\sigma v_i=w_i$ for all $i$.
    \end{quote} 
\end{definition}

\begin{lemma}\label{counting_connect}
Let $f_1,f_2,\ldots ,f_s\in \R[x_1,\ldots ,x_n]$ be polynomials whose degrees are at most $t$. Then, the number of isolated common real roots of $f_1, f_2,\ldots ,f_s$ is at most $t(2t-1)^{n-1}$. 
\end{lemma}
\begin{proof}
This is a special case of \cite{milnor_betti}.
\end{proof}

\begin{theorem}
Let \(X \subset S^d\) be an \(n\)-sized rigid \(t\)-design. Let $t'$ be $max(t,2)$ and $k$ be $(d+1)(n-d-1)$. Then, the following inequality holds:
\[ t'(2t'-1)^{k-1} \geq (n-d-1)! \]

In particular, for fixed \(t\) and \(d\), there are no rigid \(n\)-sized \(t\)-designs in \(S^d\) for large \(n\). Furthermore, there exists only a finite number of rigid \(t\)-designs in \(S^d\) up to orthogonal transformations.

\end{theorem}

\begin{proof}
Assume $X=\{v_1,v_2,\ldots, v_n\}$ is an $n$-sized rigid $t$-design of $S^d$. Let $V$ be the linear span of $X$. Without loss of generality, we can assume the linear span of $v_1,v_2,\ldots ,v_{d+1}$ is $V$. 
Let $a_{i,j}$ be the $j$-th coordinate of $v_i$. We define polynomials $f_{d+2},f_{d+3}\ldots ,f_n,g_1,\ldots ,g_m\in \mathbb{R}[x_{d+2,1},\ldots ,x_{n,d+1}]$ as follows. Here, $m=\binom{t+d+1}{d+1}$ and $P_1,P_2,\ldots, P_m$ are all monic monomials of $\R[x_1,\ldots, x_{d+1}]$ whose degrees are less than $t$; 

\[
f_i=\sum_{j=1}^d x_{i,j}^2-1 
\]
\[
g_s=\sum_{j=1}^{d+1}P_s(a_{j,1},a_{j,2},\ldots a_{j,d+1})+\sum_{j=d+2}^n P_s(x_{j,1},x_{j,2},\ldots x_{j,d+1})-\frac{\int_{S^d} P_s(y) dy}{\int_{S^d} 1 dy}.
\]
As $X$ is a $t$-design, $(a_{i,j})_{d+2\leq i\leq n,1\leq j\leq d+1}$ is a common root of $f_i$ and $g_s$. Actually, $(a_{i,j})_{d+2\leq i\leq n,1\leq j\leq d+1}$ is an isolated common root of $f_i$ and $g_s$. We will prove this.

We take $\varepsilon$ of \cref{rigid_design} for $X$. We can assume that $2\varepsilon<\min\{d(x,y)|x,y\in X,x\neq y\}$. 
Assume $(b_{i,j})_{d+2\leq i\leq n,1\leq j\leq d+1}$ is a common root of $f_i$ and $g_s$ and $\lvert b_{i,j}-a_{i,j}\rvert <\frac{\varepsilon}{\sqrt{d+1}}$. Let $w_i\in S^d$ be $(b_{i,j})_{1\leq j\leq d+1}$ for $i\geq d+2$, and $v_i$ for $i\leq d+1$.

Then, $Y=\{w_1,w_2,\ldots ,w_n\}$ is also a $t$-design and $\norm{v_i-w_i}<\varepsilon$. 
As $\varepsilon$ satisfies the condition of \cref{rigid_design}, there exists $\sigma$ such that $w_i =\sigma v_i$. In particular, $\sigma v_i=v_i$ for $i\leq d+1$. As $v_1,\ldots ,v_{d+1}$ generate $V$, $\sigma|_V =\mathrm{id}$. Therefore, we have $b_{i,j}=a_{i,j}$.
In summary, $(a_{i,j})_{d+2\leq i\leq n,1\leq j\leq d+1}$ is an isolated common root of $f_i$ and $g_s$. By symmetry, for all bijections $\tau:\{d+2,\ldots,n\}\to \{d+2,\ldots,n\}$, $(a_{\tau(i),j})_{d+2\leq i\leq n,1\leq j\leq d+1}$ are also isolated common roots. Therefore, there exist at least $(n-d-1)!$ isolated common roots of $f_i$ and $g_s$. By \cref{counting_connect}, we obtain the desired inequality.

If we fix $t$ and $d$, the inequality of this theorem does not hold for large $n$. This proves the second statement. The equivalence of the second and third statements was proved in \cite{lyubich_iff}. 
\end{proof}

\subsection*{Acknowledgement}
We would like to thank Eiichi Bannai for reviewing our paper and providing valuable comments. 

\printbibliography
\end{document}